\documentclass[reqno,A4paper,12pt]{amsart}

\usepackage{amsmath}
\usepackage{amssymb}
\usepackage{amsthm}
\usepackage{enumerate}
\usepackage{mathrsfs} 

\theoremstyle{plain}
\newtheorem{theorem}{Theorem}[section]

\theoremstyle{definition}
\newtheorem{definition}{Definition}[section]
\newtheorem{remark}{\textit{Remark}}[section]
\newtheorem{example}{\textit{Example}}[section]

\numberwithin{equation}{section}

\makeatletter
\@namedef{subjclassname@2020}{\textup{2020} Mathematics Subject Classification}
\makeatother

\begin{document}
	
\title[Statistical convergence of sequence of partial metric valued functions]%
{Statistical convergence of sequences of partial metric valued functions}
\author[P. Malik \and S. Das]%
{\textbf{Prasanta Malik$^{*}$ \and Saikat Das}}
	
\newcommand{\acr}{\newline\indent}
	
\address{Department of Mathematics, The University of Burdwan, Golapbag, Burdwan-713104, West Bengal, India.} 
\email{pmjupm@yahoo.co.in\acr 
ORCID iD: https://orcid.org/0000-0003-4987-3411}
	
\address{Department of Mathematics, The University of Burdwan, Golapbag, Burdwan-713104,
West Bengal, India.} 
\email{dassaikatsayhi@gmail.com\acr 
ORCID iD: https://orcid.org/0009-0007-7412-9244}
	
\keywords{partial metric space, convergence of sequences of functions, Cauchyness of sequences of functions, statistical pointwise convergence, statistical uniform convergence, equi-statistical convergence.\\
$\textit{*Corresponding author}$}
	
\begin{abstract}
In this paper we first introduce and study the notions of pointwise convergence and uniform convergence of sequences of partial metric valued functions. Also introducing the notions of pointwise Cauchyness and uniform Cauchyness of sequences of partial metric valued functions, we study their relationship with pointwise convergence and uniform convergence. Further we introduce and study the notions of statistical pointwise convergence and statistical uniform convergence of sequences of partial metric valued functions including their interrelationship. Also, introducing the notion of equi-statistical convergence of sequences of partial metric valued functions we study its relationship with statistical pointwise convergence and statistical uniform convergence. 
\end{abstract}
	
\maketitle

\section{\textbf{Introduction and background}}

The notion of convergence of sequences of real numbers was extended to the notion of statistical convergence by Fast \cite{Fa} and Steinhaus \cite{Sten} independently with the concept of natural density of subsets of natural numbers. A lot of works have been done in this direction, after the works of Schoenberg \cite{Sc}, Salat \cite{Sl} and Fridy \cite{Fr1}. For more primary works in this line one can see  \cite{Fr2, Ko1}. 

Using the concept of natural density, in \cite{Gokhan} G$\ddot{o}$khan et al. introduced and studied the concept of pointwise statistical convergence for sequences of real valued functions and in \cite{Gungor} G$\ddot{u}$ng$\ddot{o}$r et al. introduced and studied the notion of uniform statistical convergence for sequences of real valued functions. 

On the other hand, the concept of partial metric space was first introduced by Matthews \cite{Matt} in the year 1994, as an extension of usual metric space to address some practical problems, specially in computer programming, where ``equality $\Rightarrow$ indistancy" does not hold (see \cite{Buk}). More primary works about the notion of partial metric space, one can see \cite{Al, Buk, Ol} and many others. In \cite{Buk}, Bukatin et al. introduced and studied the notions of convergence and Cauchyness of sequences in partial metric spaces and in \cite{Nu}, Nuray extended these notions to the notions of statistical convergence and statistical Cauchyness of sequences using the concept of natural density of subsets of natural numbers.

Because of immense importance of partial metric, we first introduce the notions of pointwise convergence and uniform convergence of sequences of partial metric valued functions and study their basic properties including their interrelationship. Then we introduce the notions of pointwise and uniform Cauchyness of sequences of partial metric valued functions and study their relationship with pointwise and uniform convergence. In section 4 of this paper, we introduce the notions of statistical pointwise convergence and statistical uniform convergence of sequences of partial metric valued fucntions which naturally extend the works of \cite{Gokhan, Gungor}. We also study their basic properties including their interrelationship.
In the last section of this paper, we introduce and study the notion of equi-statistical convergence of sequences of partial metric valued functions and observe that, this notion is somewhere lying between the concepts of statistical pointwise convergence and statistical uniform convergence.


\section{\textbf{Basic Definitions and Notation}}

In this section, we recall some basic definitions and notations which are required in our study. Throughout this paper, $\mathbb{R}$ denotes the set of all real numbers, $\mathbb{R}_{\geq 0}$ denotes the set of all non-negative real numbers and $\mathcal{P}$ denotes the set of all perfect square natural numbers.

We first recall the concept of statistical convergence of sequences of real numbers and real valued functions.

\begin{definition} \cite{Fr1, Sl}
Let $\mathcal{A}$ be a subset of $\mathbb N$. For each $n \in \mathbb{N}$, let $d_{n}(\mathcal{A}) = |\mathcal{A} \cap \{1, 2, \dots, n\}|/n$, called the $n^{th}$ partial density of $\mathcal{A}$, where $|\mathcal{A}|$ denote the cardinality of the set $\mathcal{A}$. We say that, the set $\mathcal{A}$ has natural density $d(\mathcal{A})$ if the limit $\displaystyle{\lim_{n\rightarrow \infty}}d_{n}(\mathcal{A})$ exists finitely and
\begin{align*}
d(\mathcal{A}) = \displaystyle{\lim_{n\rightarrow \infty}} d_{n}(\mathcal{A}).
\end{align*}
\end{definition}

\begin{definition} \cite{Fr1, Sl}
Let $\{x_{n}\}_{n\in \mathbb N}$ be a sequence of real numbers. Then $\{x_{n}\}_{n\in \mathbb N}$ is said to be statistically convergent to $x_{o} (\in \mathbb R)$, if for any $\epsilon > 0, ~d(\mathcal{A}(\epsilon))= 0$, where
\begin{align*}
\mathcal A(\epsilon)= \left\{n\in \mathbb N: \left|x_{n} - x_{o}\right|\geq \epsilon\right\}.
\end{align*}
In this case, we write $st-\displaystyle{\lim_{n\rightarrow \infty}}x_{n}= x_{o}.$
\end{definition}

\begin{definition} \cite{Gokhan}
A sequence of real valued functions $\{f_{k}\}_{k\in\mathbb{N}}$ defined on $\mathcal{S} ~(\subset \mathbb{R})$ is said to pointwise statistically convergent to a real valued function $f$ on  $\mathcal{S}$ if, for each $x \in \mathcal{S}$ and for every $\epsilon > 0$,
\begin{align*}
\displaystyle{\lim_{n\rightarrow \infty}} \frac{1}{n}|\{k \leq n: |f_{k}(x) - f(x)| \geq \epsilon\}| = 0.
\end{align*}
In this case, we write $f_{k} \stackrel{st}{\rightarrow} f$ on $\mathcal{S}$.
\end{definition} 

\begin{definition} \cite{Gungor}
A sequence of real valued functions $\{f_{k}\}_{k\in\mathbb{N}}$ defined on $\mathcal{S} ~(\subset \mathbb{R})$ is said to be uniformly statistically convergent to a real valued function $f$ on  $\mathcal{S}$ if, for every $\epsilon > 0$,
\begin{align*}
\displaystyle{\lim_{n\rightarrow \infty}} \frac{1}{n}|\{k \leq n: |f_{k}(x) - f(x)| \geq \epsilon ~\mbox{for every}~ x \in \mathcal{S}\}| = 0.
\end{align*}
In this case, we write $f_{k} \stackrel{st}{\rightrightarrows} f$ on $\mathcal{S}$.
\end{definition}

Following the line of  S. G. Matthews \cite{Matt}, we now recall the concept of partial metric space. 

\begin{definition} \cite{Matt} \label{definition1}
Let $\mathcal{X}$ be a non-empty set and a mapping $\wp: \mathcal{X \times X} \rightarrow \mathbb{R}$ is said to be a partial metric on $\mathcal{X}$ if for any $x, y, z \in \mathcal{X}$, the following conditions are satisfied:
\begin{itemize}
\item[$(\wp1)$] $x = y$ if and only if $\wp(x,x) = \wp(x,y) = \wp(y,y)$;
\item[$(\wp2)$] $0 \leq \wp(x,x) \leq \wp(x,y)$;
\item[$(\wp3)$] $\wp(x,y) = \wp(y,x)$;
\item[$(\wp4)$] $\wp(x,y) \leq \wp(x,z) + \wp(z,y) - \wp(z,z)$. 
\end{itemize}
	
The pair $(\mathcal{X},\wp)$ is called a partial metric space.
\end{definition}

\begin{remark}
Every metric space is a partial metric space but not conversely. To show this we consider the following example.
\end{remark}

\begin{example}
Let $\wp: \mathbb{R \times R}\rightarrow \mathbb{R}$ be a mapping defined as follows:
\begin{align*}
\wp(x,y) = \left|x - y\right| + 5, ~ x, y \in \mathbb{R}.
\end{align*}
Then $\wp$ satisfies all the conditions $(\wp1)$ to $(\wp4)$ of the Definition \ref{definition1}. Hence $(\mathbb{R},\wp)$ is a partial metric space. But $\wp(1,1) = 5 \neq 0$. Therefore $\wp$ is not a metric on $\mathbb{R}$.
\end{example}

\begin{definition} \cite{Buk}
A sequence $\{x_{n}\}_{n\in \mathbb N}$ in a partial metric space $(\mathcal{X},\wp)$ is said to be convergent to a point $x_{o} (\in \mathcal X)$, if 	$\displaystyle{\lim_{n\rightarrow \infty}} \wp(x_{n},x_{o}) = \displaystyle{\lim_{n\rightarrow \infty}} \wp(x_{n},x_{n}) = \wp(x_{o},x_{o})$, i.e. if for every $\epsilon > 0$, there exists $k_{o} \in \mathbb N$, such that
\begin{align*}
&\left|\wp(x_{n}, x_{o})- \wp(x_{o}, x_{o})\right|< \epsilon, \\
\mbox{and}~ & \left|\wp(x_{n}, x_{n})- \wp(x_{o}, x_{o})\right|< \epsilon,  ~\forall~ n\geq k_{o}. 
\end{align*}
In this case, we write $\displaystyle{\lim_{n\rightarrow \infty}}x_{n} = x_{o}$.
\end{definition}

\begin{definition} \cite{Buk}
A sequences $\{x_{n}\}_{n\in \mathbb{N}}$ in a partial metric space $(\mathcal{X},\wp)$ is said to be Cauchy, if there exists a non-negative real number $L$ such that $\displaystyle{\lim_{m,n\rightarrow \infty}}\wp(x_{m},x_{n}) = L$, i.e. if for every $\epsilon > 0$, there exists $k = k(\epsilon) \in \mathbb{N}$, such that
\begin{align*}
\left|\wp(x_{m},x_{n}) - L\right| < \epsilon, ~\forall~ m, n \geq k.
\end{align*}
\end{definition}

\begin{definition} \cite{Nu}
Let $(\mathcal{X},\wp)$ be a partial metric space and  $\{x_{n}\}_{n\in \mathbb{N}}$ be a sequence in $\mathcal{X}$. Then $\{x_{n}\}_{n\in \mathbb{N}}$ is said to be statistically convergent to a point $x_{o} ~(\in \mathcal{X})$, if 
$st-\displaystyle{\lim_{n\rightarrow \infty}}\wp(x_{n},x_{o}) = st-\displaystyle{\lim_{n\rightarrow \infty}}\wp(x_{n},x_{n}) = \wp(x_{o},x_{o})$ i.e. if for any $\epsilon > 0$, there exists $\mathcal{A}(\epsilon) \subset \mathbb{N}$ such that $d(\mathcal{A}) = 0$ and
\begin{align*}
&\left|\wp(x_{n}, x_{o})- \wp(x_{o}, x_{o})\right|< \epsilon, ~\mbox{and}\\
&\left|\wp(x_{n}, x_{n})- \wp(x_{o}, x_{o})\right|< \epsilon,  ~\mbox{for all}~ n \in \mathbb{N} - \mathcal{A}.
\end{align*}
	
In this case, we write $st-\displaystyle{\lim_{n\rightarrow \infty}} x_{n} = x_{o}$. 
\end{definition}


\section{\textbf{Convergence and Cauchyness of sequences of partial metric valued functions}}

In this section we introduce the notions of convergence and Cauchyness of sequences of partial metric valued functions. Throughout the paper, $(\mathcal{X},\wp)$ denotes a partial metric space, $\mathcal{D}$ be a non-empty set and all functions are $(\mathcal{X},\wp)$ valued functions on $\mathcal{D}$, unless otherwise mentioned.


\begin{definition}
	Let $(\mathcal{X},\wp)$ be a partial metric space, $\mathcal{D}$ be a non-empty set and for each $n \in \mathbb{N}, ~f_{n}: \mathcal{D}\rightarrow \mathcal{X}$ be a function. Then $\{f_{n}\}_{n\in \mathbb{N}}$ is said to be pointwise convergent to a function $f: \mathcal{D}\rightarrow \mathcal{X}$, if for each $x \in \mathcal{D}$, the sequence $\{f_{n}(x)\}_{n\in \mathbb{N}}$ converges to $f(x)$, i.e. $\displaystyle{\lim_{n\rightarrow \infty}}\wp(f_{n}(x),f(x)) = \displaystyle{\lim_{n\rightarrow \infty}}\wp(f_{n}(x),f_{n}(x)) = \wp(f(x),f(x))$ i.e. for each $\epsilon > 0$, there exists $k = k(\epsilon,x) \in \mathbb{N}$, such that
	\begin{align*}
		& |\wp(f_{n}(x),f(x)) - \wp(f(x),f(x))| < \epsilon, \\
		\mbox{and}~ & |\wp(f_{n}(x),f_{n}(x)) - \wp(f(x),f(x))| < \epsilon, ~\mbox{for all}~ n \geq k.
	\end{align*}
	In this case we say that, $\{f_{n}\}_{n\in\mathbb{N}}$ converges to $f$ on $\mathcal{D}$ and we write $f_{n} \rightarrow f$ on $\mathcal{D}$, where $f$ is called a pointwise limit of $\{f_{n}\}_{n\in \mathbb{N}}$. 
\end{definition}

\begin{remark}
	In the next Theorem we show that, pointwise limit function of a sequence $\{f_{n}\}_{n\in\mathbb{N}}$ of partial metric valued functions is uniquely determined.
\end{remark}

\begin{theorem}
	Let $(\mathcal{X},\wp)$ be a partial metric space and $\{f_{n}\}_{n\in\mathbb{N}}$ be a sequence of $(\mathcal{X},\wp)$ valued functions on a non-empty set
	$\mathcal{D}$. If the sequence $\{f_{n}\}_{n\in\mathbb{N}}$ converges on $\mathcal{D}$, then the limit function is unique.
\end{theorem}

\begin{proof}
Let $f_{n}\rightarrow f$ and $f_{n}\rightarrow g$ on $\mathcal{D}$. Fix $x_{o} \in \mathcal{D}$. Then, 
\begin{align*}
& \displaystyle{\lim_{n\rightarrow \infty}} \wp(f_{n}(x_{o}),f(x_{o})) = \displaystyle{\lim_{n\rightarrow \infty}} \wp(f_{n}(x_{o}),f_{n}(x_{o})) = \wp(f(x_{o}),f(x_{o})) \\
\mbox{and}~ & \displaystyle{\lim_{n\rightarrow \infty}} \wp(f_{n}(x_{o}),g(x_{o})) = \displaystyle{\lim_{n\rightarrow \infty}} \wp(f_{n}(x_{o}),f_{n}(x_{o})) = \wp(g(x_{o}),g(x_{o})).
\end{align*}
Since the limit of a convergent sequence of real numbers is uniquely determined, so we must have $\wp(f(x_{o}),f(x_{o})) = \wp(g(x_{o}),g(x_{o}))$. Now,
\begin{align*}
&\wp(f(x_{o}),g(x_{o})) \leq \wp(f(x_{o}),f_{n}(x_{o})) + \wp(f_{n}(x_{o}),g(x_{o})) - \wp(f_{n}(x_{o}),f_{n}(x_{o})), ~\mbox{for all}~ n \in \mathbb{N}.
\end{align*}
Taking limit on both sides as $n \rightarrow \infty$, we get
\begin{align*}
& \wp(f(x_{o}),g(x_{o})) \leq \wp(f(x_{o}),f(x_{o})) + \wp(g(x_{o}),g(x_{o})) - \wp(g(x_{o}),g(x_{o})) \\
\Rightarrow~ & \wp(f(x_{o}),g(x_{o})) \leq \wp(f(x_{o}),f(x_{o})) \leq \wp(f(x_{o}),g(x_{o})) \\
\Rightarrow~ & \wp(f(x_{o}),g(x_{o})) = \wp(f(x_{o}),f(x_{o})).
\end{align*}
Thus, 
\begin{align*}
& \wp(f(x_{o}),f(x_{o})) =  \wp(f(x_{o}),g(x_{o})) = \wp(g(x_{o}),g(x_{o})) \\
\Rightarrow~ & f(x_{o}) = g(x_{o}). 
\end{align*}
Since, $x_{o} \in \mathcal{D}$ is arbitrary, so $f = g$ on $\mathcal{D}$.
\end{proof}

\begin{definition}
Let $(\mathcal{X},\wp)$ be a partial metric space, $\mathcal{D}$ be a non-empty set and for each $n \in \mathbb{N}, ~f_{n}: \mathcal{D}\rightarrow \mathcal{X}$ be a function. Then $\{f_{n}\}_{n\in \mathbb{N}}$ is said to be uniformly convergent to the function $f: \mathcal{D}\rightarrow \mathcal{X}$ on $\mathcal{D}$, if for each $\epsilon > 0$, there exists $k = k(\epsilon) \in \mathbb{N}$, such that
\begin{align*}
&\forall~ x \in \mathcal{D}, |\wp(f_{n}(x),f(x)) - \wp(f(x),f(x))| < \epsilon ~\mbox{and}~ |\wp(f_{n}(x),f_{n}(x)) - \wp(f(x),f(x))| < \epsilon, \\
&\mbox{for all}~ n \geq k.
\end{align*}
In this case, we write $f_{n} \rightrightarrows f$ on $\mathcal{D}$ and $f$ is called the uniform limit of $\{f_{n}\}_{n\in\mathbb{N}}$.
\end{definition}

\begin{remark}
It is clear that, uniform convergence implies pointwise convergence but not conversely. 
\end{remark}
Next theorem gives a necessary and sufficient condition for uniform convergence of sequences of partial metric valued functions.

\begin{theorem}
Let $(\mathcal{X},\wp)$ be a partial metric space, $\mathcal{D}$ be a non-empty set and for each $n \in \mathbb{N}, ~f_{n}: \mathcal{D}\rightarrow \mathcal{X}$ be a function. Then $\{f_{n}\}_{n\in \mathbb{N}}$ is uniformly convergent to the function $f: \mathcal{D}\rightarrow \mathcal{X}$ on $\mathcal{D}$ if and only if
$\displaystyle{\lim_{n\rightarrow\infty}\sup_{x\in \mathcal{D}}}|\wp(f_{n}(x),f(x)) - \wp(f(x),f(x))| = \displaystyle{\lim_{n\rightarrow\infty}\sup_{x\in \mathcal{D}}}|\wp(f_{n}(x),f_{n}(x)) - \wp(f(x),f(x))|= 0$.
\end{theorem}

\begin{proof}
The proof is straightforward and so omitted. 
\end{proof}

Now we introduce the notions of pointwise Cauchyness and uniform Cauchyness of sequences of partial metric valued functions.

\begin{definition}
Let $(\mathcal{X},\wp)$ be a partial metric space and $\{f_{n}\}_{n\in\mathbb{N}}$ be a sequence of $(\mathcal{X},\wp)$ valued functions defined on a non-empty set $\mathcal{D}$.
Then $\{f_{n}\}_{n\in\mathbb{N}}$ is said to be pointwise Cauchy on $\mathcal{D}$, if for each $x \in \mathcal{D}$, the sequence $\{\wp(f_{m}(x),f_{n}(x))\}_{m,n\in \mathbb{N}}$ converges to a real number $l = l(x) \geq 0$, i.e. if for each $\epsilon > 0$, there exists $k = k(\epsilon,x) \in \mathbb{N}$ such that,
\begin{align*}
|\wp(f_{m}(x),f_{n}(x)) - l| < \epsilon, ~\mbox{for all}~ m, n \geq k.
\end{align*} 
\end{definition}

\begin{definition}
Let $(\mathcal{X},\wp)$ be a partial metric space, $\mathcal{D}$ be a non-empty set and for each $n \in \mathbb{N}, ~f_{n}: \mathcal{D}\rightarrow \mathcal{X}$ be a function. Then $\{f_{n}\}_{n\in\mathbb{N}}$ is said to be uniformly Cauchy on $\mathcal{D}$, if the convergence of the sequence $\{\wp(f_{m}(x),f_{n}(x))\}_{m,n\in \mathbb{N}}$ is independent of $x ~(\in \mathcal{D})$, i.e. if for each $\epsilon > 0$, there exists $N = N(\epsilon) \in \mathbb{N}$ such that,
\begin{align*}
\sup_{x \in \mathcal{D}} |\wp(f_{m}(x),f_{n}(x)) - \wp(f_{j}(x),f_{k}(x))| < \epsilon, ~\mbox{for all}~ m \geq j \geq N, ~n \geq k \geq N.
\end{align*}
\end{definition}

\begin{remark}
Clearly every Uniform Cauchy sequence of partial metric valued functions is pointwise Cauchy but not conversely.	
\end{remark}

We now study the relationship between the notions of pointwise convergence and uniform convergence of sequences of partial metric valued functions with the notions of Cauchyness. 

\begin{theorem} \label{Theorem3.4}
Let $(\mathcal{X},\wp)$ be a partial metric space, $\mathcal{D}$ be a non-empty set and for each $n \in \mathbb{N}, ~f_{n}: \mathcal{D}\rightarrow \mathcal{X}$ be a function. If the sequence $\{f_{n}\}_{n\in\mathbb{N}}$ converges pointwise on $\mathcal{D}$, then it is pointwise Cauchy on $\mathcal{D}$.
\end{theorem}

\begin{proof}
Let the sequence $\{f_{n}\}_{n\in\mathbb{N}}$ be pointwise convergent to the function $f: \mathcal{D}\rightarrow (\mathcal{X},\wp)$ on $\mathcal{D}$. Let $x_{o} \in \mathcal{D}$. Then we have
\begin{align*}
\displaystyle{\lim_{n\rightarrow \infty}}\wp(f_{n}(x_{o}),f(x_{o})) = \displaystyle{\lim_{n\rightarrow \infty}}\wp(f_{n}(x_{o}),f_{n}(x_{o})) = \wp(f(x_{o}),f(x_{o})).
\end{align*}
Let $\epsilon > 0$ be given. Then there exists $k = k(\epsilon,x_{o}) \in \mathbb{N}$, such that
\begin{align*}
& |\wp(f_{n}(x_{o}),f(x_{o})) - \wp(f(x_{o}),f(x_{o}))| < \frac{\epsilon}{4}, \\
\mbox{and}~ & |\wp(f_{n}(x_{o}),f(x_{o})) - \wp(f_{n}(x_{o}),f_{n}(x_{o}))| < \frac{\epsilon}{4}, ~\mbox{for all}~ n \geq k.
\end{align*}
Then for all $m, n \geq k$,
\begin{align*}
&|\wp(f_{m}(x_{o}),f_{n}(x_{o})) - \wp(f(x_{o}),f(x_{o}))| \leq |\wp(f_{m}(x_{o}),f(x_{o})) - \wp(f(x_{o}),f(x_{o}))| + \\ 
& |\wp(f_{n}(x_{o}),f(x_{o})) - \wp(f(x_{o}),f(x_{o}))| + |\wp(f_{m}(x_{o}),f(x_{o})) - \wp(f_{m}(x_{o}),f_{m}(x_{o}))| \\
&+ |\wp(f_{n}(x_{o}),f(x_{o})) - \wp(f_{n}(x_{o}),f_{n}(x_{o}))| < \epsilon.
\end{align*}  
So, the sequence $\{\wp(f_{m}(x_{o}),f_{n}(x_{o}))\}_{m,n\in\mathbb{N}}$ is convergent to $\wp(f(x_{o}),f(x_{o}))$. Hence the sequence $\{f_{n}\}_{n\in\mathbb{N}}$ is pointwise Cauchy on $\mathcal{D}$.
\end{proof}

\begin{theorem} \label{Theorem3.5}
Let $(\mathcal{X},\wp)$ be a partial metric space, $\mathcal{D}$ be a non-empty set and for each $n \in \mathbb{N}, ~f_{n}: \mathcal{D}\rightarrow \mathcal{X}$ be a function. If the sequence $\{f_{n}\}_{n\in\mathbb{N}}$ is uniformly convergent on $\mathcal{D}$, then it is uniformly Cauchy on $\mathcal{D}$.
\end{theorem}

\begin{proof}
The proof is straightforward, so is omitted.
\end{proof}

\begin{remark} 
The converse implications of the Theorem \ref{Theorem3.4} and Theorem \ref{Theorem3.5} are not true. To show this, we consider the following example.
\end{remark}

\begin{example}
Let $\wp: \mathbb{R \times R}\rightarrow \mathbb{R}$ be defined as follows:
\begin{equation*}
\wp(x, y) = |x - y| + 1, ~ x, y \in \mathbb{R}.
\end{equation*}
Then $(\mathbb{R},\wp)$ is a non-metric partial metric space. Let $\mathcal{D} = (0,1]$ and for each $n \in \mathbb{N}, ~f_{n}: \mathcal{D}\rightarrow (\mathbb{R},\wp)$ be defined as
\begin{equation*}
f_{n}(x) = \frac{1}{n}\sin(\frac{x}{n}), ~ x \in \mathcal{D}.
\end{equation*}
Then $\{f_{n}\}_{n\in\mathbb{N}}$ is a sequence of $(\mathbb{R},\wp)$ valued functions on $\mathcal{D}$.

We see that, for each $x \in \mathcal{D}$, the sequence $\{f_{n}(x)\}_{n\in\mathbb{N}}$ does not converge to any point in $\mathbb{R}$ i.e. $\displaystyle{\lim_{n\rightarrow \infty}} \wp(f_{n}(x),l) \neq \wp(l,l)$ for any $x \in \mathcal{D}$ and for any $l \in \mathbb{R}$. Therefore, the sequence $\{f_{n}\}_{n\in \mathbb{N}}$ is not pointwise convergent and so, not convergent uniformly on $\mathcal{D}$.

Let $\epsilon > 0$ be given. Then there exists $N = N(\epsilon) \in \mathbb{N}$ such that for all $x \in \mathcal{D}$,
\begin{align*}
|\wp(f_{m}(x),f_{n}(x)) - \wp(f_{j}(x),f_{k}(x))| 
&\leq  |\wp(f_{m}(x),f_{n}(x)) - 1| + |\wp(f_{j}(x),f_{k}(x)) - 1| \\
&<  \frac{\epsilon}{4} + \frac{\epsilon}{4} = \frac{\epsilon}{2}, ~\forall~ m \geq j \geq N, n
\geq k \geq N 
\end{align*}
\begin{align*}
\Rightarrow~  \sup_{x \in \mathcal{D}} |\wp(f_{m}(x),f_{n}(x)) - \wp(f_{j}(x),f_{k}(x))| \leq \frac{\epsilon}{2} < \epsilon, ~\forall~ m \geq j \geq N, n \geq k \geq N. 
\end{align*}
This shows that, the sequence $\{f_{n}\}_{n\in\mathbb{N}}$ is Uniformly Cauchy and so pointwise Cauchy without being Uniformly convergent.
\end{example}

However, converse of the Theorem \ref{Theorem3.5} holds under certain conditions as mentioned in the next Theorem.

\begin{theorem}
Let $(\mathcal{X},\wp)$ be a partial metric space, $\mathcal{D}$ be a non-empty set and for each $n \in \mathbb{N}, ~f_{n}: \mathcal{D}\rightarrow (\mathcal{X},\wp)$ be a function. If the sequence $\{f_{n}\}_{n\in\mathbb{N}}$ is uniformly Cauchy on $\mathcal{D}$ and has a subsequence $\{f_{n_{i}}\}_{i\in\mathbb{N}}$, converging uniformly to the function $f: \mathcal{D}\rightarrow (\mathcal{X},\wp)$, then it is uniformly convergent to $f$ on $\mathcal{D}$.
\end{theorem}

\begin{proof}
The proof is trivial, so is omitted.
\end{proof}


\section{\textbf{Statistical pointwise and Statistical uniform convergence of sequences of partial metric valued functions}}

Here we introduce the notions of statistical pointwise convergence and statistical uniform convergence of sequences of partial metric valued functions.

\begin{definition}
Let $(\mathcal{X},\wp)$ be a partial metric space, $\mathcal{D}$ be a non-empty set and for each $n \in \mathbb{N}, ~f_{n}: \mathcal{D}\rightarrow (\mathcal{X},\wp)$ be a function. Then $\{f_{n}\}_{n\in \mathbb{N}}$ is said to be statistically pointwise convergent to the function $f: \mathcal{D}\rightarrow (\mathcal{X},\wp)$ on $\mathcal{D}$, if for each $x \in \mathcal{D}$, $st-\displaystyle{\lim_{n\rightarrow \infty}}\wp(f_{n}(x),f(x)) = st-\displaystyle{\lim_{n\rightarrow \infty}}\wp(f_{n}(x),f_{n}(x)) = \wp(f(x),f(x))$, i.e. if for each $x \in \mathcal{D}$ and for each $\epsilon > 0$, there exists $\mathcal{A} = \mathcal{A}(\epsilon,x) \subset \mathbb{N}$ such that $d(\mathcal{A}) = 0$ and
\begin{align*}
& |\wp(f_{n}(x),f(x)) - \wp(f(x),f(x))| < \epsilon, \\
\mbox{and}~ & |\wp(f_{n}(x),f_{n}(x)) - \wp(f(x),f(x))| < \epsilon, ~\forall~ n \in \mathbb{N} - \mathcal{A}.
\end{align*}
In this case, we write $f_{n} \stackrel{st}{\rightarrow} f$ and $f$ is called the statistical pointwise limit of $\{f_{n}\}_{n\in \mathbb{N}}$ on $\mathcal{D}$.
\end{definition}

\begin{remark} 
If $(\mathcal{X},\wp)$ is a partial metric space and if $\{f_{n}\}_{n\in\mathbb{N}}$ is a sequence of $(\mathcal{X},\wp)$ valued functions defined on a non-empty set $\mathcal{D}$, converging pointwise to the function $f: \mathcal{D}\rightarrow (\mathcal{X},\wp)$ on $\mathcal{D}$, then clearly $\{f_{n}\}_{n\in\mathbb{N}}$ is statistically pointwise convergent to $f$ on $\mathcal{D}$. But the converse is not true. 
\end{remark}

\begin{definition}
Let $(\mathcal{X},\wp)$ be a partial metric space and $\{f_{n}\}_{n\in\mathbb{N}}$ be a sequence of $(\mathcal{X},\wp)$ valued functions defined on a non-empty set $\mathcal{D}$. Then $\{f_{n}\}_{n\in\mathbb{N}}$ is said to be statistically uniformly convergent to the function $f: \mathcal{D}\rightarrow (\mathcal{X},\wp)$ on $\mathcal{D}$, if for each $\epsilon > 0$, there exists $\mathcal{A} = \mathcal{A}(\epsilon) \subset \mathbb{N}$ such that $d(\mathcal{A}) = 0$ and 
\begin{align*}
&\forall~ x \in \mathcal{D}, |\wp(f_{n}(x),f(x)) - \wp(f(x),f(x))| < \epsilon ~\mbox{and}~ |\wp(f_{n}(x),f_{n}(x)) - \wp(f(x),f(x))| < \epsilon, \\
&\mbox{for all}~ n \in \mathbb{N} - \mathcal{A}.
\end{align*}
In this case, we write $f_{n} \stackrel{st}{\rightrightarrows} f$ and $f$ is said to be the statistical uniform limit of the sequence $\{f_{n}\}_{n\in \mathbb{N}}$ on $\mathcal{D}$.
\end{definition}

\begin{remark}	If $(\mathcal{X},\wp)$ is a partial metric space and if $\{f_{n}\}_{n\in\mathbb{N}}$ is a sequence of $(\mathcal{X},\wp)$ valued functions on a non-empty set $\mathcal{D}$, converging uniformly to the function $f: \mathcal{D}\rightarrow (\mathcal{X},\wp)$, then $\{f_{n}\}_{n\in\mathbb{N}}$ is statistically uniformly convergent to $f$ on $\mathcal{D}$. But the converse is not true.
\end{remark}

\begin{remark}
Let $(\mathcal{X},\wp)$ be a partial metric space and $\{f_{n}\}_{n\in \mathbb{N}}$ be a sequence of $(\mathcal{X},\wp)$ valued functions on a non-empty set $\mathcal{D}$. If the sequence $\{f_{n}\}_{n\in \mathbb{N}}$ is statistically uniformly convergent to the function $f: \mathcal{D}\rightarrow (\mathcal{X},\wp)$ on $\mathcal{D}$, then it is statistically pointwise convergent to $f$ on $\mathcal{D}$. but not conversely. 
\end{remark}

\begin{theorem} \label{Theorem4.1}
Let $(\mathcal{X},\wp)$ be a partial metric space and $\{f_{n}\}_{n\in \mathbb{N}}$ be a sequence of $(\mathcal{X},\wp)$ valued functions on a non-empty set $\mathcal{D}$. Then $\{f_{n}\}_{n\in \mathbb{N}}$ is statistically uniformly convergent to the function $f: \mathcal{D}\rightarrow (\mathcal{X},\wp)$ on $\mathcal{D}$ if and only if $st-\displaystyle{\lim_{n\rightarrow\infty}\sup_{x\in \mathcal{D}}}|\wp(f_{n}(x),f(x)) - \wp(f(x),f(x))| = st-\displaystyle{\lim_{n\rightarrow\infty}\sup_{x\in \mathcal{D}}}|\wp(f_{n}(x),f_{n}(x)) - \wp(f(x),f(x))|= 0$.
\end{theorem} 

\begin{proof}
Let the sequence $\{f_{n}\}_{n\in \mathbb{N}}$ be statistically uniformly convergent to $f$ on $\mathcal{D}$. Let $\epsilon > 0$ be given. Then there exists $\mathcal{A} = \mathcal{A}(\epsilon) \subset \mathbb{N}$ such that $d(\mathcal{A}) = 0$ and for all $x \in \mathcal{D}$,
\begin{align*}
&|\wp(f_{n}(x),f(x)) - \wp(f(x),f(x))| < \frac{\epsilon}{2},
~|\wp(f_{n}(x),f_{n}(x)) - \wp(f(x),f(x))| < \frac{\epsilon}{2}, ~\forall~ n \in \mathbb{N} - \mathcal{A} \\
\Rightarrow~ & \sup_{x\in \mathcal{D}}|\wp(f_{n}(x),f(x)) - \wp(f(x),f(x))| < \epsilon,
~\sup_{x\in \mathcal{D}}|\wp(f_{n}(x),f_{n}(x)) - \wp(f(x),f(x))| < \epsilon, ~\forall~ n \in \mathbb{N} - \mathcal{A} \\
\Rightarrow~ & st-\displaystyle{\lim_{n\rightarrow\infty}\sup_{x\in \mathcal{D}}}|\wp(f_{n}(x),f(x)) - \wp(f(x),f(x))| = st-\displaystyle{\lim_{n\rightarrow\infty}\sup_{x\in \mathcal{D}}}|\wp(f_{n}(x),f_{n}(x)) - \\
&\wp(f(x),f(x))|= 0.
\end{align*}
Conversely, let $st-\displaystyle{\lim_{n\rightarrow\infty}\sup_{x\in \mathcal{D}}}|\wp(f_{n}(x),f(x)) - \wp(f(x),f(x))| = st-\displaystyle{\lim_{n\rightarrow\infty}\sup_{x\in \mathcal{D}}}|\wp(f_{n}(x),f_{n}(x)) - \wp(f(x),f(x))|= 0$. Let $\epsilon > 0$ be given. Then 
\begin{align*}
& d(\{n\in \mathbb{N}: \sup_{x\in \mathcal{D}}|\wp(f_{n}(x),f(x)) - \wp(f(x),f(x))| \geq \epsilon\}) = d(\{n\in \mathbb{N}: \sup_{x\in \mathcal{D}}|\wp(f_{n}(x),f_{n}(x)) \\
&- \wp(f(x),f(x))| \geq \epsilon\}) = 0.
\end{align*}
Let $\mathcal{B}_{1} = \{n\in \mathbb{N}: \sup_{x\in \mathcal{D}} |\wp(f_{n}(x),f(x)) - \wp(f(x),f(x))| \geq \epsilon\}$ and $\mathcal{B}_{2} = \{n\in \mathbb{N}: \sup_{x\in \mathcal{D}}|\wp(f_{n}(x),f_{n}(x)) - \wp(f(x),f(x))| \geq \epsilon\}$. Let $\mathcal{A} = \mathcal{B}_{1} \cup \mathcal{B}_{2}$. Since $d(\mathcal{B}_{1}) = d(\mathcal{B}_{2}) = 0$, so $d(\mathcal{A}) = 0$. Thus, for all $x \in \mathcal{D}$
\begin{align*}
&|\wp(f_{n}(x),f(x)) - \wp(f(x),f(x))| \leq \sup_{x\in \mathcal{D}} |\wp(f_{n}(x),f(x)) - \wp(f(x),f(x))| < \epsilon ~\mbox{and}\\
&|\wp(f_{n}(x),f_{n}(x)) - \wp(f(x),f(x))| \leq \sup_{x\in \mathcal{D}} |\wp(f_{n}(x),f_{n}(x)) - \wp(f(x),f(x))| < \epsilon, ~\forall 
n \in ~\mathbb{N} - \mathcal{A}.
\end{align*} 
This gives, $f_{n} \stackrel{st}{\rightrightarrows} f$, i.e. the sequence $\{f_{n}\}_{n\in \mathbb{N}}$ is statistically uniformly convergent to $f$ on $\mathcal{D}$.
\end{proof}


\section{\textbf{Equi-statistical convergence of partial metric valued sequence of functions}}

We now introduce the notion of equi-statistical convergence of sequences of partial metric valued functions which lies somewhere between statistical pointwise convergence and statistical uniform convergence. 

\begin{definition}
Let $(\mathcal{X},\wp)$ be a partial metric space, $\mathcal{D}$ be a non-empty set and for each $n \in \mathbb{N}, ~f_{n}: \mathcal{D}\rightarrow \mathcal{X}$ be a function. Then $\{f_{n}\}_{n\in\mathbb{N}}$ is said to be equi-statistically convergent to the function $f: \mathcal{D}\rightarrow \mathcal{X}$ on $\mathcal{D}$, if for each $\epsilon > 0$, the sequences  $\{F_{j,\epsilon}\}_{j\in\mathbb{N}}$ and $\{H_{j,\epsilon}\}_{j\in\mathbb{N}}$ of real valued functions on $\mathcal{D}$ are uniformly convergent to the zero function on $\mathcal{D}$, where for $x \in \mathcal{D}$,
\begin{align*}
&F_{j,\epsilon}(x) = d_{j}(\{n\in\mathbb{N}:  |\wp(f_{n}(x),f(x)) - \wp(f(x),f(x))| \geq \epsilon\}) \\
\mbox{and}~ &H_{j,\epsilon}(x) = d_{j}(\{n\in\mathbb{N}: |\wp(f_{n}(x),f_{n}(x)) - \wp(f(x),f(x))| \geq \epsilon\})
\end{align*}
In this case, we write $st_{eq}-\displaystyle{\lim_{n\rightarrow \infty}} f_{n} = f$ and $f$ is called the equi-statistical limit of $\{f_{n}\}_{n\in\mathbb{N}}$ on $\mathcal{D}$.
\end{definition}

\begin{theorem}
Let $(\mathcal{X},\wp)$ be a partial metric space, $\mathcal{D}$ be a non-empty set and for each $n \in \mathbb{N}, ~f_{n}: \mathcal{D}\rightarrow \mathcal{X}$ be a function. If the sequence $\{f_{n}\}_{n\in\mathbb{N}}$ is statistically uniformly convergent to the function $f: \mathcal{D}\rightarrow \mathcal{X}$ on $\mathcal{D}$, then it is equi-statistically convergent to $f$ on $\mathcal{D}$.
\end{theorem}

\begin{proof}
Let the sequence $\{f_{n}\}_{n\in\mathbb{N}}$ of $(\mathcal{X},\wp)$ valued functions be statistically uniformly convergent to $f$ on $\mathcal{D}$. Let $\epsilon_{o} > 0$. Then there exists $\mathcal{A} = \mathcal{A}(\epsilon_{o}) \subset \mathbb{N}$ such that $d(\mathcal{A}) = 0$ and for all $x \in \mathcal{D}$, 
\begin{align*}
&|\wp(f_{n}(x),f(x)) - \wp(f(x),f(x))| < \epsilon_{o}, \\
\mbox{and}~ &|\wp(f_{n}(x),f_{n}(x)) - \wp(f(x),f(x))| < \epsilon_{o}, ~\forall~ n \in \mathbb{N} - \mathcal{A}.
\end{align*}
So, for each $x \in \mathcal{D}$, we have
\begin{align*}
&\{n\in\mathbb{N}:  |\wp(f_{n}(x),f(x)) - \wp(f(x),f(x))| \geq \epsilon_{o}\} \subset \mathcal{A} \\
\mbox{and}~ &\{n\in\mathbb{N}:  |\wp(f_{n}(x),f_{n}(x)) - \wp(f(x),f(x))| \geq \epsilon_{o}\} \subset \mathcal{A} \\
\Rightarrow~ &F_{j,\epsilon_{o}}(x) \leq d_{j}(\mathcal{A}) ~\mbox{and}~ H_{j,\epsilon_{o}}(x) \leq d_{j}(\mathcal{A}), ~\forall~ j \in \mathbb{N}.
\end{align*}
Let $\sigma > 0$ be given. Since $d(\mathcal{A}) = 0$, there exists $j_{o} \in \mathbb{N}$ such that
\begin{align*}
& d_{j}(\mathcal{A}) < \sigma, ~\mbox{for all}~ j \geq j_{o} \\
\Rightarrow~ \mbox{for all}~ x \in \mathcal{D},~ & F_{j,\epsilon_{o}}(x) < \sigma ~\mbox{and}~ H_{j,\epsilon_{o}}(x) < \sigma, ~\forall~ j \geq j_{o}.
\end{align*}
Therefore, both the sequences $\{F_{j,\epsilon_{o}}\}_{j\in\mathbb{N}}$ and $\{H_{j,\epsilon_{o}}\}_{j\in\mathbb{N}}$ are uniformly convergent to the zero function on $\mathcal{D}$. Hence the sequence $\{f_{n}\}_{n\in\mathbb{N}}$ is equi-statistically convergent to the function $f$ on $\mathcal{D}$.
\end{proof}

\begin{remark}
The converse of the above theorem is not true. 
\end{remark}


\begin{thebibliography}{9}
	
\bibitem{Al} 
I. Altun, F. Sola, H. Simsek, \textit{Generalized contractions on partial metric spaces}, Topology Appl., \textbf{157(18)} (2010), 2778--2785.
	
	
	
\bibitem{Buk} 
M. Bukatin, R. Kopperman, S. Matthews, H. Pajoohesh, \textit{Partial metric spaces}, 
Amer. Math. Monthly, \textbf{116(8)} (2009), 708--718.
	
\bibitem{Fa} 
H. Fast, \textit{Sur la convergence statistique}, Colloq. Math., \textbf{2(34)} (1951), 241--244.
	
\bibitem{Fre} 
M. F$\acute{r}$echet, \textit{Sur L'$\acute{e}$cart de Deux Courbes et Sur Les Courbes Limites}, Trans. Amer. Math. Soc., \textbf{6(4)} (1905), 435--449.
	
\bibitem{Fr1} 
J. A. Fridy, \textit{On statistical convergence}, Analysis., \textbf{5(4)} (1985), 301--314.
	
\bibitem{Fr2}  
J.A. Fridy, \textit{Statistical limit points}, Proc. Amer. Math. Soc., \textbf{118(4)} (1993), 1187--1192.
	
\bibitem{Gokhan} 
A. G$\ddot{o}$khan, M. G$\ddot{u}$ng$\ddot{o}$r, \textit{On pointwise statistical convergence}, Indian J. Pure Appl. Math., \textbf{33(9)} (2002), 1379--1384.
	
\bibitem{Gungor}
M. G$\ddot{u}$ng$\ddot{o}$r, A. G$\ddot{o}$khan, \textit{On uniform statistical convergence}, Int. J. Pure Appl. Math., \textbf{19(1)} (2005), 17--24.
	
\bibitem{Ko1} 
P. Kostyrko, M. Ma$\check{c}$aj, T. $\check{S}$al$\acute{a}$t, O. Strauch, \textit{On statistical limit points}, Proc. Amer. Math. Soc., \textbf{129(9)} (2001), 2647--2654.
	
	
\bibitem{Matt} 
S. G. Matthews, \textit{Partial metric topology}, Ann. N. Y. Acad. Sci., \textbf{728(1)} (1994), 183--197. 
	
\bibitem{Mur1} 
M. Mursaleen, O. H. H. Edely, \textit{Statistical convergence of double sequences}, J. Math. Anal. Appl., \textbf{288} (2003), 223--231. 
	
\bibitem{Nu} 
F. Nuray, \textit{Statistical convergence in Partial metric spaces}, Korean J. Math., \textbf{30(1)} (2022), 155--160.
	
\bibitem{Ol} 
S. Oltra, O. Valero, \textit{Banach's fixed point theorem for partial metric spaces}, Rend. Istit. Mat. Univ. Trieste, \textbf{36} (2004), 17--26.
	
\bibitem{Sl} 
T. $\check{S}$al$\acute{a}$t, \textit{On Statistical Convergent Sequences of Real Numbers}, Math. Slovaca, \textbf{30(2)} (1980), 139--150.
	
\bibitem{Sc} 
I. J. Schoenberg, \textit{The Integrability of certain functions and related summability methods}, Amer. Math. Monthly, \textbf{66(5)} (1959), 361--375.
	
\bibitem{Sten}
H. Steinhaus, \textit{Sur la convergence ordinaire et la convergence asymptotique}, Colloq. math., \textbf{2(1)} (1951), 73--74.
	
\end{thebibliography}
\end{document}